\documentclass[12pt]{amsart}       
\usepackage{txfonts}
\usepackage{amssymb}
\usepackage{eucal}
\usepackage{graphicx}
\usepackage{amsmath}
\usepackage{amscd}
\usepackage[all]{xy}           
\usepackage{tikz}
\usepackage{amsfonts,latexsym}
\usepackage{xspace}
\usepackage{epsfig}
\usepackage{float}
\usepackage{axodraw}
\usepackage{color}
\usepackage{fancybox}
\usepackage{colordvi}
\usepackage{multicol}
\usepackage{colordvi}
\usepackage[colorlinks,final,backref=page,hyperindex,hypertex]{hyperref}
\usepackage[active]{srcltx} 

\newtheorem{theorem}{Theorem}[section]
\newtheorem{prop}[theorem]{Proposition}
\theoremstyle{definition}
\newtheorem{defn}[theorem]{Definition}
\newtheorem{lemma}[theorem]{Lemma}

\newtheorem{prop-def}{Proposition-Definition}[section]
\newtheorem{coro-def}{Corollary-Definition}[section]

\newtheorem{exam}{Example}[section]

\newcommand{\nc}{\newcommand}
\nc{\tred}[1]{\textcolor{red}{#1}}
\nc{\tblue}[1]{\textcolor{blue}{#1}}
\nc{\tgreen}[1]{\textcolor{green}{#1}}
\nc{\tpurple}[1]{\textcolor{purple}{#1}}
\nc{\btred}[1]{\textcolor{red}{\bf #1}}
\nc{\btblue}[1]{\textcolor{blue}{\bf #1}}
\nc{\btgreen}[1]{\textcolor{green}{\bf #1}}
\nc{\btpurple}[1]{\textcolor{purple}{\bf #1}}
\nc{\NN}{{\mathbb N}}
\nc{\ncsha}{{\mbox{\cyr X}^{\mathrm NC}}} \nc{\ncshao}{{\mbox{\cyr
X}^{\mathrm NC}_0}}

\renewcommand{\frak}{\mathfrak}

\newcommand{\efootnote}[1]{}

\renewcommand{\textbf}[1]{}

\newcommand{\delete}[1]{}

\nc{\mlabel}[1]{\label{#1}}  
\nc{\mcite}[1]{\cite{#1}}  
\nc{\mref}[1]{\ref{#1}}  
\nc{\mbibitem}[1]{\bibitem{#1}} 

\delete{
\nc{\mlabel}[1]{\label{#1}  
{\hfill \hspace{1cm}{\small\tt{{\ }\hfill(#1)}}}}
\nc{\mcite}[1]{\cite{#1}{\small{\tt{{\ }(#1)}}}}  
\nc{\mref}[1]{\ref{#1}{{\tt{{\ }(#1)}}}}  
\nc{\mbibitem}[1]{\bibitem[\bf #1]{#1}} 
}

\nc{\opa}{\ast} \nc{\opb}{\odot} \nc{\op}{\bullet} \nc{\pa}{\frakL}
\nc{\arr}{\rightarrow} \nc{\lu}[1]{(#1)} \nc{\mult}{\mrm{mult}}
\nc{\diff}{\mathfrak{Diff}}
\nc{\opc}{\sharp}\nc{\opd}{\natural}
\nc{\ope}{\circ}
\nc{\dpt}{\mathrm{d}}
\nc{\diam}{alternating\xspace}
\nc{\Diam}{Alternating\xspace}
\nc{\cdiam}{alternating\xspace}
\nc{\Cdiam}{Alternating\xspace}
\nc{\AW}{\mathcal{A}}
\nc{\rba}{Rota-Baxter algebra\xspace}

\nc{\ari}{\mathrm{ar}}

\nc{\lef}{\mathrm{lef}}

\nc{\Sh}{\mathrm{ST}}

\nc{\Cr}{\mathrm{Cr}}

\nc{\st}{{Schr\"oder tree}\xspace}
\nc{\sts}{{Schr\"oder trees}\xspace}

\nc{\vertset}{\Omega} 

\nc{\assop}{\quad \begin{picture}(5,5)(0,0)
\line(-1,1){10}
\put(-2.2,-2.2){$\bullet$}
\line(0,-1){10}\line(1,1){10}
\end{picture} \quad \smallskip}

\nc{\operator}{\begin{picture}(5,5)(0,0)
\line(0,-1){6}
\put(-2.6,-1.8){$\bullet$}
\line(0,1){9}
\end{picture}}

\nc{\idx}{\begin{picture}(6,6)(-3,-3)
\put(0,0){\line(0,1){6}}
\put(0,0){\line(0,-1){6}}
 \end{picture}}

\nc{\pb}{{\mathrm{pb}}}
\nc{\Lf}{{\mathrm{Lf}}}

\nc{\lft}{{left tree}\xspace}
\nc{\lfts}{{left trees}\xspace}

\nc{\fat}{{fundamental averaging tree}\xspace}

\nc{\fats}{{fundamental averaging trees}\xspace}
\nc{\avt}{\mathrm{Avt}}

\nc{\rass}{{\mathit{RAss}}}

\nc{\aass}{{\mathit{AAss}}}

\nc{\vin}{{\mathrm Vin}}    
\nc{\lin}{{\mathrm Lin}}    
\nc{\inv}{\mathrm{I}n}
\nc{\gensp}{V} 
\nc{\genbas}{\mathcal{V}} 
\nc{\bvp}{V_P}     
\nc{\gop}{{\,\omega\,}}     

\nc{\bin}[2]{ (_{\stackrel{\scs{#1}}{\scs{#2}}})}  
\nc{\binc}[2]{ \left (\!\! \begin{array}{c} \scs{#1}\\
    \scs{#2} \end{array}\!\! \right )}  
\nc{\bincc}[2]{  \left ( {\scs{#1} \atop
    \vspace{-1cm}\scs{#2}} \right )}  
\nc{\bs}{\bar{S}} \nc{\cosum}{\sqsubset} \nc{\la}{\longrightarrow}
\nc{\rar}{\rightarrow} \nc{\dar}{\downarrow} \nc{\dprod}{**}
\nc{\dap}[1]{\downarrow \rlap{$\scriptstyle{#1}$}}
\nc{\md}{\mathrm{dth}} \nc{\uap}[1]{\uparrow
\rlap{$\scriptstyle{#1}$}} \nc{\defeq}{\stackrel{\rm def}{=}}
\nc{\disp}[1]{\displaystyle{#1}} \nc{\dotcup}{\
\displaystyle{\bigcup^\bullet}\ } \nc{\gzeta}{\bar{\zeta}}
\nc{\hcm}{\ \hat{,}\ } \nc{\hts}{\hat{\otimes}}
\nc{\barot}{{\otimes}} \nc{\free}[1]{\bar{#1}}
\nc{\uni}[1]{\tilde{#1}} \nc{\hcirc}{\hat{\circ}} \nc{\lleft}{[}
\nc{\lright}{]} \nc{\lc}{\lfloor} \nc{\rc}{\rfloor}
\nc{\curlyl}{\left \{ \begin{array}{c} {} \\ {} \end{array}
    \right .  \!\!\!\!\!\!\!}
\nc{\curlyr}{ \!\!\!\!\!\!\!
    \left . \begin{array}{c} {} \\ {} \end{array}
    \right \} }
\nc{\longmid}{\left | \begin{array}{c} {} \\ {} \end{array}
    \right . \!\!\!\!\!\!\!}
\nc{\onetree}{\bullet} \nc{\ora}[1]{\stackrel{#1}{\rar}}
\nc{\ola}[1]{\stackrel{#1}{\la}}
\nc{\ot}{\otimes} \nc{\mot}{{{\boxtimes\,}}}
\nc{\otm}{\overline{\boxtimes}} \nc{\sprod}{\bullet}
\nc{\scs}[1]{\scriptstyle{#1}} \nc{\mrm}[1]{{\rm #1}}
\nc{\margin}[1]{\marginpar{\rm #1}}   
\nc{\dirlim}{\displaystyle{\lim_{\longrightarrow}}\,}
\nc{\invlim}{\displaystyle{\lim_{\longleftarrow}}\,}
\nc{\mvp}{\vspace{0.3cm}} \nc{\tk}{^{(k)}} \nc{\tp}{^\prime}
\nc{\ttp}{^{\prime\prime}} \nc{\svp}{\vspace{2cm}}
\nc{\vp}{\vspace{8cm}} \nc{\proofbegin}{\noindent{\bf Proof: }}
\nc{\proofend}{$\blacksquare$ \vspace{0.3cm}}
\nc{\modg}[1]{\!<\!\!{#1}\!\!>}
\nc{\intg}[1]{F_C(#1)} \nc{\lmodg}{\!
<\!\!} \nc{\rmodg}{\!\!>\!}
\nc{\cpi}{\widehat{\Pi}}
\nc{\sha}{{\mbox{\cyr X}}}  
\nc{\shap}{{\mbox{\cyrs X}}} 
\nc{\shpr}{\diamond}    
\nc{\shp}{\ast} \nc{\shplus}{\shpr^+}
\nc{\shprc}{\shpr_c}    
\nc{\msh}{\ast} \nc{\zprod}{m_0} \nc{\oprod}{m_1}
\nc{\vep}{\varepsilon} \nc{\labs}{\mid\!} \nc{\rabs}{\!\mid}
\nc{\sqmon}[1]{\langle #1\rangle}

\nc{\mmbox}[1]{\mbox{\ #1\ }} \nc{\dep}{\mrm{dep}} \nc{\fp}{\mrm{FP}}
\nc{\rchar}{\mrm{char}} \nc{\End}{\mrm{End}} \nc{\Fil}{\mrm{Fil}}
\nc{\Mor}{Mor\xspace} \nc{\gmzvs}{gMZV\xspace}
\nc{\gmzv}{gMZV\xspace} \nc{\mzv}{MZV\xspace}
\nc{\mzvs}{MZVs\xspace} \nc{\Hom}{\mrm{Hom}} \nc{\id}{\mrm{id}}
\nc{\im}{\mrm{im}} \nc{\incl}{\mrm{incl}} \nc{\map}{\mrm{Map}}
\nc{\mchar}{\rm char} \nc{\nz}{\rm NZ} \nc{\supp}{\mathrm Supp}

\nc{\Alg}{\mathbf{Alg}} \nc{\Bax}{\mathbf{Bax}} \nc{\bff}{\mathbf f}
\nc{\bfk}{{\bf k}} \nc{\bfone}{{\bf 1}} \nc{\bfx}{\mathbf x}
\nc{\bfy}{\mathbf y}
\nc{\base}[1]{\bfone^{\otimes ({#1}+1)}} 
\nc{\Cat}{\mathbf{Cat}}

\nc{\detail}{\marginpar{\bf More detail}
    \noindent{\bf Need more detail!}
    \svp}
\nc{\Int}{\mathbf{Int}} \nc{\Mon}{\mathbf{Mon}}
\nc{\rbtm}{{shuffle }} \nc{\rbto}{{Rota-Baxter }}
\nc{\remarks}{\noindent{\bf Remarks: }} \nc{\Rings}{\mathbf{Rings}}
\nc{\Sets}{\mathbf{Sets}} \nc{\wtot}{\widetilde{\odot}}
\nc{\wast}{\widetilde{\ast}} \nc{\bodot}{\bar{\odot}}
\nc{\bast}{\bar{\ast}} \nc{\hodot}[1]{\odot^{#1}}
\nc{\hast}[1]{\ast^{#1}} \nc{\mal}{\mathcal{O}}
\nc{\tet}{\tilde{\ast}} \nc{\teot}{\tilde{\odot}}
\nc{\oex}{\overline{x}} \nc{\oey}{\overline{y}}
\nc{\oez}{\overline{z}} \nc{\oef}{\overline{f}}
\nc{\oea}{\overline{a}} \nc{\oeb}{\overline{b}}
\nc{\weast}[1]{\widetilde{\ast}^{#1}}
\nc{\weodot}[1]{\widetilde{\odot}^{#1}} \nc{\hstar}[1]{\star^{#1}}
\nc{\lae}{\langle} \nc{\rae}{\rangle}
\nc{\lf}{\lfloor}
\nc{\rf}{\rfloor}

\def\ta1{{\scalebox{0.25}{ 
\begin{picture}(12,12)(38,-38)
\SetWidth{0.5} \SetColor{Black} \Vertex(45,-33){5.66}
\end{picture}}}}

\def\tb2{{\scalebox{0.25}{ 
\begin{picture}(12,42)(38,-38)
\SetWidth{0.5} \SetColor{Black} \Vertex(45,-3){5.66}
\SetWidth{1.0} \Line(45,-3)(45,-33) \SetWidth{0.5}
\Vertex(45,-33){5.66}
\end{picture}}}}

\def\tc3{{\scalebox{0.25}{ 
\begin{picture}(12,72)(38,-38)
\SetWidth{0.5} \SetColor{Black} \Vertex(45,27){5.66}
\SetWidth{1.0} \Line(45,27)(45,-3) \SetWidth{0.5}
\Vertex(45,-33){5.66} \SetWidth{1.0} \Line(45,-3)(45,-33)
\SetWidth{0.5} \Vertex(45,-3){5.66}
\end{picture}}}}

\def\td31{{\scalebox{0.25}{ 
\begin{picture}(42,42)(23,-38)
\SetWidth{0.5} \SetColor{Black} \Vertex(45,-3){5.66}
\Vertex(30,-33){5.66} \Vertex(60,-33){5.66} \SetWidth{1.0}
\Line(45,-3)(30,-33) \Line(60,-33)(45,-3)
\end{picture}}}}

\def\te4{{\scalebox{0.25}{ 
\begin{picture}(12,102)(38,-8)
\SetWidth{0.5} \SetColor{Black} \Vertex(45,57){5.66}
\Vertex(45,-3){5.66} \Vertex(45,27){5.66} \Vertex(45,87){5.66}
\SetWidth{1.0} \Line(45,57)(45,27) \Line(45,-3)(45,27)
\Line(45,57)(45,87)
\end{picture}}}}

\def\tf41{{\scalebox{0.25}{ 
\begin{picture}(42,72)(38,-8)
\SetWidth{0.5} \SetColor{Black} \Vertex(45,27){5.66}
\Vertex(45,-3){5.66} \SetWidth{1.0} \Line(45,27)(45,-3)
\SetWidth{0.5} \Vertex(60,57){5.66} \SetWidth{1.0}
\Line(45,27)(60,57) \SetWidth{0.5} \Vertex(75,27){5.66}
\SetWidth{1.0} \Line(75,27)(60,57)
\end{picture}}}}

\def\tg42{{\scalebox{0.25}{ 
\begin{picture}(42,72)(8,-8)
\SetWidth{0.5} \SetColor{Black} \Vertex(45,27){5.66}
\Vertex(45,-3){5.66} \SetWidth{1.0} \Line(45,27)(45,-3)
\SetWidth{0.5} \Vertex(15,27){5.66} \Vertex(30,57){5.66}
\SetWidth{1.0} \Line(15,27)(30,57) \Line(45,27)(30,57)
\end{picture}}}}

\def\th43{{\scalebox{0.25}{ 
\begin{picture}(42,42)(8,-8)
\SetWidth{0.5} \SetColor{Black} \Vertex(45,-3){5.66}
\Vertex(15,-3){5.66} \Vertex(30,27){5.66} \SetWidth{1.0}
\Line(15,-3)(30,27) \Line(45,-3)(30,27) \Line(30,27)(30,-3)
\SetWidth{0.5} \Vertex(30,-3){5.66}
\end{picture}}}}

\def\thII43{{\scalebox{0.25}{ 
\begin{picture}(72,57) (68,-128)
    \SetWidth{0.5}
    \SetColor{Black}
    \Vertex(105,-78){5.66}
    \SetWidth{1.5}
    \Line(105,-78)(75,-123)
    \Line(105,-78)(105,-123)
    \Line(105,-78)(135,-123)
    \SetWidth{0.5}
    \Vertex(75,-123){5.66}
    \Vertex(105,-123){5.66}
    \Vertex(135,-123){5.66}
  \end{picture}
  }}}

\def\thj44{{\scalebox{0.25}{ 
\begin{picture}(42,72)(8,-8)
\SetWidth{0.5} \SetColor{Black} \Vertex(30,57){5.66}
\SetWidth{1.0} \Line(30,57)(30,27) \SetWidth{0.5}
\Vertex(30,27){5.66} \SetWidth{1.0} \Line(45,-3)(30,27)
\SetWidth{0.5} \Vertex(45,-3){5.66} \Vertex(15,-3){5.66}
\SetWidth{1.0} \Line(15,-3)(30,27)
\end{picture}}}}

\def\ti5{{\scalebox{0.25}{ 
\begin{picture}(12,132)(23,-8)
\SetWidth{0.5} \SetColor{Black} \Vertex(30,117){5.66}
\SetWidth{1.0} \Line(30,117)(30,87) \SetWidth{0.5}
\Vertex(30,87){5.66} \Vertex(30,57){5.66} \Vertex(30,27){5.66}
\Vertex(30,-3){5.66} \SetWidth{1.0} \Line(30,-3)(30,27)
\Line(30,27)(30,57) \Line(30,87)(30,57)
\end{picture}}}}

\def\tj51{{\scalebox{0.25}{ 
\begin{picture}(42,102)(53,-38)
\SetWidth{0.5} \SetColor{Black} \Vertex(61,27){4.24}
\SetWidth{1.0} \Line(75,57)(90,27) \Line(60,27)(75,57)
\SetWidth{0.5} \Vertex(90,-3){5.66} \Vertex(60,27){5.66}
\Vertex(75,57){5.66} \Vertex(90,-33){5.66} \SetWidth{1.0}
\Line(90,-33)(90,-3) \Line(90,-3)(90,27) \SetWidth{0.5}
\Vertex(90,27){5.66}
\end{picture}}}}

\def\tk52{{\scalebox{0.25}{ 
\begin{picture}(42,102)(23,-8)
\SetWidth{0.5} \SetColor{Black} \Vertex(60,57){5.66}
\Vertex(45,87){5.66} \SetWidth{1.0} \Line(45,87)(60,57)
\SetWidth{0.5} \Vertex(30,57){5.66} \SetWidth{1.0}
\Line(30,57)(45,87) \SetWidth{0.5} \Vertex(30,-3){5.66}
\SetWidth{1.0} \Line(30,-3)(30,27) \SetWidth{0.5}
\Vertex(30,27){5.66} \SetWidth{1.0} \Line(30,57)(30,27)
\end{picture}}}}

\def\tl53{{\scalebox{0.25}{ 
\begin{picture}(42,102)(8,-8)
\SetWidth{0.5} \SetColor{Black} \Vertex(30,57){5.66}
\Vertex(30,27){5.66} \SetWidth{1.0} \Line(30,57)(30,27)
\SetWidth{0.5} \Vertex(30,87){5.66} \SetWidth{1.0}
\Line(30,27)(45,-3) \SetWidth{0.5} \Vertex(15,-3){5.66}
\SetWidth{1.0} \Line(15,-3)(30,27) \Line(30,57)(30,87)
\SetWidth{0.5} \Vertex(45,-3){5.66}
\end{picture}}}}

\def\tm54{{\scalebox{0.25}{ 
\begin{picture}(42,72)(8,-38)
\SetWidth{0.5} \SetColor{Black} \Vertex(30,-3){5.66}
\SetWidth{1.0} \Line(30,27)(30,-3) \Line(30,-3)(45,-33)
\SetWidth{0.5} \Vertex(15,-33){5.66} \SetWidth{1.0}
\Line(15,-33)(30,-3) \SetWidth{0.5} \Vertex(45,-33){5.66}
\SetWidth{1.0} \Line(30,-33)(30,-3) \SetWidth{0.5}
\Vertex(30,-33){5.66} \Vertex(30,27){5.66}
\end{picture}}}}

\def\tn55{{\scalebox{0.25}{ 
\begin{picture}(42,72)(8,-38)
\SetWidth{0.5} \SetColor{Black} \Vertex(15,-33){5.66}
\Vertex(45,-33){5.66} \Vertex(30,27){5.66} \SetWidth{1.0}
\Line(45,-33)(45,-3) \SetWidth{0.5} \Vertex(45,-3){5.66}
\Vertex(15,-3){5.66} \SetWidth{1.0} \Line(30,27)(45,-3)
\Line(15,-3)(30,27) \Line(15,-3)(15,-33)
\end{picture}}}}

\nc{\QQ}{{\mathbb Q}}
\nc{\RR}{{\mathbb R}} \nc{\ZZ}{{\mathbb Z}}

\nc{\cala}{{\mathcal A}} \nc{\calb}{{\mathcal B}}
\nc{\calc}{{\mathcal C}}
\nc{\cald}{{\mathcal D}} \nc{\cale}{{\mathcal E}}
\nc{\calf}{{\mathcal F}} \nc{\calg}{{\mathcal G}}
\nc{\calh}{{\mathcal H}} \nc{\cali}{{\mathcal I}}
\nc{\call}{{\mathcal L}} \nc{\calm}{{\mathcal M}}
\nc{\caln}{{\mathcal N}} \nc{\calo}{{\mathcal O}}
\nc{\calp}{{\mathcal P}} \nc{\calr}{{\mathcal R}}
\nc{\cals}{{\mathcal S}} \nc{\calt}{{\mathcal T}}
\nc{\calu}{{\mathcal U}} \nc{\calw}{{\mathcal W}} \nc{\calk}{{\mathcal K}}
\nc{\calx}{{\mathcal X}} \nc{\CA}{\mathcal{A}}

\nc{\fraka}{{\mathfrak a}} \nc{\frakA}{{\mathfrak A}}
\nc{\frakb}{{\mathfrak b}} \nc{\frakB}{{\mathfrak B}}
\nc{\frakD}{{\mathfrak D}} \nc{\frakF}{\mathfrak{F}}
\nc{\frakf}{{\mathfrak f}} \nc{\frakg}{{\mathfrak g}}
\nc{\frakH}{{\mathfrak H}} \nc{\frakL}{{\mathfrak L}}
\nc{\frakM}{{\mathfrak M}} \nc{\bfrakM}{\overline{\frakM}}
\nc{\frakm}{{\mathfrak m}} \nc{\frakP}{{\mathfrak P}}
\nc{\frakN}{{\mathfrak N}} \nc{\frakp}{{\mathfrak p}}
\nc{\frakS}{{\mathfrak S}} \nc{\frakT}{\mathfrak{T}}
\nc{\frakX}{{\mathfrak X}} \nc{\frakx}{\mathfrak{x}}

\nc{\BS}{\mathbb{S
}}

\font\cyr=wncyr10 \font\cyrs=wncyr7
\nc{\li}[1]{\textcolor{red}{Li:#1}}
\nc{\tian}[1]{\textcolor{blue}{Tianjie: #1}}
\nc{\xing}[1]{\textcolor{purple}{Xing: #1}}

\nc{\ID}{\mathfrak{I}} \nc{\lbar}[1]{\overline{#1}}
\nc{\bre}{{\rm b}} \nc{\sd}{\cals} \nc{\rb}{\rm RB}
\nc{\A}{\rm angularly decorated\xspace} \nc{\LL}{\rm L}
\nc{\w}{\rm wid} \nc{\arro}[1]{#1}
\nc{\ver}{\rm ver}

\begin{document}

\title{Bialgebra and Hopf algebra structures on free Rota-Baxter algebras}
%
\author{Xing Gao}
\address{School of Mathematics and Statistics, Key Laboratory of Applied Mathematics and Complex Systems, Lanzhou University, Lanzhou, Gansu 730000, P.\,R. China}
         \email{gaoxing@lzu.edu.cn}

\author{Li Guo}
\address{Department of Mathematics and Computer Science,
         Rutgers University,
         Newark, NJ 07102, USA}
\email{liguo@rutgers.edu}

\author{Tianjie Zhang}
\address{Department of Mathematics, Lanzhou University, Lanzhou, Gansu 730000, P.\,R. China}
         \email{tjzhangmath@aliyun.com}

\date{\today}
\begin{abstract}
In this paper, we obtain a canonical factorization of basis elements in free Rota-Baxter algebras built on bracketed words. This canonical factorization is applied to give a coalgebra structure on the free Rota-Baxter algebras. Together with the Rota-Baxter algebra multiplication, this coproduct gives a bialgebra structure on the free Rota-Baxter algebra of rooted forests. When the weight of the Rota-Baxter algebra is zero, we further obtain a Hopf algebra structure.
\end{abstract}

\subjclass[2010]{16W99,16S10,16T10,16W70,05E99}

\keywords{Rota-Baxter algebra, words, bracketed words, factorization, coalgebra, bialgebra, Hopf algebra}

\maketitle

\tableofcontents

\setcounter{section}{0}

\allowdisplaybreaks

\section{Introduction}

The concepts of a Hopf algebra and a bialgebra originated from topology study and are built from the combination of an algebra structure and coalgebra structure on the same linear space. Their study has a long history, a very rich theory and broad applications in mathematics and physics~\mcite{Ab,CK,Sw}. An important class of Hopf algebras is built from free objects in various contexts. Classical examples include free associative algebras and the enveloping algebras of Lie algebras. Other examples have appeared in recent years. For example the free objects in the category of dendriform algebras of Loday and of tridendriform algebras of Loday and Ronco~\cite{LR}, and more generally, from splitting of the associativity~\cite{BBGN,Lo3,PBG}. Under the commutativity condition, the above free objects recover the well-known Hopf algebras of shuffles and quasi-shuffles~\cite{Ho} most notable for their applications in multiple zeta values. It is worth noting that the Connes-Kreimer Hopf algebra of rooted trees also has its algebra structure from a free object, namely free operated algebra~\cite{Guop}.

Adding to this class of Hopf algebras, we obtain in this paper bialgebras and Hopf algebras from free Rota-Baxter algebras on bracketed words~\cite{EG3,G6}.

A Rota-Baxter algebra (first known as a Baxter algebra) is an associative algebra $R$ equipped with a linear
operator $P$ that generalizes the integral operator in analysis. More precisely,
\begin{equation}
P(u)P(v)=P(uP(v))+P(P(u)v)+\lambda P(uv)\ \text{ for all } u, v\in R.
\mlabel{eq:rbo}
\end{equation}
Here $\lambda$, called the weight of the Rota-Baxter operator $P$, is a prefixed element in the base ring of the algebra $R$.

Originated from the probability study of G. Baxter~\cite{Ba} and continuing the early work of well-known mathematicians such as Atkinson~\cite{FV}, Cartier~\cite{Ca} and Rota~\cite{Ro}, the theory of Rota-Baxter algebra has been well developed in recent years, with wide range of applications, in quantum field theory,
operads, Hopf algebras, commutative algebra, combinatorics and number theory~\cite{Ag,AM,EG,EG3,EGK,G2,G4,G6,Gub,Guop,G9,ML,MY,MN}. The authors of~\cite{JZ,ML} introduced the concept of a Rota-Baxter coalgebra by dualizing the Rota-Baxter operator, as well as the multiplication.

Free commutative Rota-Baxter algebras, as constructed by mixable shuffles in~\cite{G4}, has been~\cite{EG1} equipped with a Hopf algebra structure by means of the shuffle and quasi-shuffle algebras which constitutes a major part of a free commutative Rota-Baxter algebra. This paper deals with free Rota-Baxter algebras without the commutativity condition, constructed by bracketed words~\cite{EG1,Gub}. More precisely we show that the noncommutative Rota-Baxter algebra is a bialgebra, and is a Hopf algebra when the weight of the Rota-Baxter algebra is zero.

For the purpose of defining the coproduct, we give a canonical factorization of a basis in the free Rota-Baxter algebra. It states that any basis element of the free Rota-Baxter algebra can be uniquely factorized as the product of indecomposable elements in a particular form. This factorization allows us to reduce the constructions and proofs for the bialgebra to the case of indecomposable bracketed words.

Here is the layout of this paper. In Section~\ref{sec:RBHOPHAL}, we recall the construction of a free Rota-Baxter algebra whose basis is given by bracketed words. We then obtain a unique factorization of such bracketed words (Proposition~\mref{pp:cdiam}). Making use of this factorization, we define a coproduct on a free Rota-Baxter algebra in Section~\ref{sec:bial} and verify the axioms for a bialgebra. The theorem on the bialgebra structure for free Rota-Baxter algebras is stated in Theorem~\ref{thm:main} and the proof is carried out in the next two subsections of Section~\ref{sec:bial}. In the final Section~\mref{sec:hopf}, we show that a free Rota-Baxter algebra of weight zero is a Hopf algebra and give some discussion on the non-zero weight case.

\smallskip

\noindent
{\bf Convention. } Throughout this paper, all algebras are taken to be unitary over a unitary commutative ring $\bfk$ with identity $1_\bfk$ which we often abbreviated as 1.

\section{Unique factorization in free Rota-Baxter algebras}
\label{sec:RBHOPHAL}
In this section we first recall the construction of free Rota-Baxter algebras by bracketed words. We then provide a unique factorization, called the \cdiam factorization, of basis elements in a
free Rota-Baxter algebra. For more details see~\cite{EG1,EG3,Gub}.
This factorization will be applied in the next section to obtain a
bialgebra structure on free Rota-Baxter algebras.

\begin{defn}
{\rm
A {\bf free Rota-Baxter algebra on a set $X$} is an \rba $F(X)$ with a Rota-Baxter
operator $P_X$ and a set map $j_X: X\to F(X)$ such
that, for any \rba $R$ and any set map
$f:X\to R$, there is a unique \rba homomorphism
$\free{f}: F(X)\to R$ such that $\free{f}\circ j_X=f$.
}
\end{defn}

We next recall the recursive construction of a canonical $\bfk$-basis of the free \rba as certain words, called {Rota-Baxter words}, on the set $X$.
These words are obtained from free operated monoid on $X$, whose construction we now recall.

For any set $Y$, let $M(Y)$ denote the free monoid generated by $Y$ and let
$\lc Y\rc$ denote the set $\{ \lc y\rc \, |\, y\in Y\}.$
Thus $\lc Y\rc$ is a set indexed by $Y$ but disjoint with $Y$. We recursively define a direct system
$$\{\frakM_n, \uni{i}_{n,n+1}: \frakM_n\to \frakM_{n+1} \}$$
of free monoids with injective transition maps.
We first let
$\frakM_0:=M(X)$ and then define
$\frakM_1:=M(X\cup \lc M(X)\rc)$
with $i_{0,1}$ being the natural injections
\begin{align*}
 i_{0,1}:& \frakM_0=M(X) \hookrightarrow
    \frakM_1=M(X\cup \lc \frakM_0\rc)
\end{align*}
We identify $\frakM_0$ with their images in $\frakM_1$. In particular, $1\in \frakM_0$ is sent to
$1\in \frakM_1$.

Inductively assume that $\frakM_{n}$ has been defined
for $n\geq 1$, we define
\begin{equation}
 \frakM_{n+1}:=M(X\cup \lc\frakM_{n}\rc ).
 \mlabel{eq:frakm}
 \end{equation}
Further assume that the embedding
$$i_{n-1,n}: \frakM_{n-1} \to \frakM_{n}$$
has been obtained. Then we have the injection
$$  \lc\frakM_{n-1}\rc \hookrightarrow
    \lc \frakM_{n} \rc.$$
Thus by the freeness of
$\frakM_{n}=M(X\cup \lc\frakM_{n-1}\rc)$, we have
\begin{eqnarray*}
\frakM_{n} &=& M(X\cup \lc\frakM_{n-1}\rc)\hookrightarrow
    M(X\cup \lc \frakM_{n}\rc) =\frakM_{n+1}.
\end{eqnarray*}
We finally define the monoid
$$ \frakM(X):=\dirlim \frakM_n$$ with identity $1$.
Then~\mcite{Guop,Gub}
$\bfk\,\frakM(X)$ is the free operated $\bfk$-algebra on $X$.

Let $Y$ and $Z$ be subsets of $\frakM(X)$. Define the {\bf alternating products} of $Y$ and $Z$ by
\index{alternating product}
\allowdisplaybreaks{
\begin{eqnarray}
\Lambda(Y,Z)&=&\Big( \bigcup_{r\geq 1} \big (Y\lc Z\rc \big)^r \Big) \bigcup
    \Big(\bigcup_{r\geq 0} \big (Y\lc Z\rc \big)^r  Y\Big) \notag \\
&& \bigcup \Big( \bigcup_{r\geq 1} \big( \lc Z\rc Y \big )^r \Big)
 \bigcup \Big( \bigcup_{r\geq 0} \big (\lc Z\rc Y\big )^r \lc Z\rc \Big) \bigcup \Big\{1\Big\}.
\mlabel{eq:uwords}
\end{eqnarray}
}
They are subsets of $\frakM(X)$.

Recursively define $$\frak X_{0}:=M(X)\,\text{ and }\,\frak X_{n}:=\Lambda(\frak X_{0},\frak X_{n-1}),n\geq1.$$
Thus $\frak X_{0}\subseteq\cdots\subseteq\frak X_{n}\subseteq\cdots.$

The elements in $\frak X_{\infty}:=\dirlim\frak X_{n} =\cup_{n\geq 0} \frak X_{n}$
are called {\bf Rota-Baxter bracketed words}(RBWs).
For a RBW $w\in\frak X_{\infty}$, we call $\dep(w):=\min\{n\mid w\in \frak X_{n}\}$ the
{\bf depth} of $w$.

\begin{lemma}\cite{Gub}
Every RBW $x\neq1$ has a unique decomposition $x=x_{1}\cdots x_{b},$
where $x_{i},1\leq i\leq b$, is alternatively in the free semigroup $S(X)$ or in $\lc\frak X_{\infty}\rc:=\{\lc w\rc\,|\,w\in \frakX_\infty\}$.
\end{lemma}

Let
$\ncsha(X):=\bfk\frak X_{\infty}$
be the \bfk-module spanned by $\frak X_{\infty}$.  To make it into a Rota-Baxter algebra, we will equip it with a product $\diamond$ and a Rata-Baxter operator $P_{X}$~\mcite{EG1,Gub}.

Let $w,w'$ be two base elements in $\frak X_{\infty}$. Define $P(w)=\lc w\rc$.
Next we define $w\diamond w'$ inductively on the sum $n:=\dep(w)+\dep(w')\geq 0$.
If $n=0$, then $w,w'\in\frak X_{0}=M(X)$ and define $w\diamond w':=xx'$, the concatenation in $M(X)$.
Suppose that $w\diamond w'$ have been defined by for all RBWs $w,w'\in\frak X_{\infty}$
with $n\leq k$ for a $k\geq0$ and consider RBWs $w,w'\in\frak X_{\infty}$ with $n=k+1$.
First assume that $\bre(w)=\bre(w')=1$. Then $w$ and $w'$ are in $S(X)\subset\frak X_{0}$ or
$\lc\frak X_{\infty}\rc$ and can not be both in $S(X)$ since $n=k+1\geq1$. We accordingly define
\begin{equation}
\begin{aligned}
w\diamond w'=\left\{\begin{array}{ll}
ww',&\text{ if }w\in S(X) \text{ or }\,w'\in S(X),\\
\lc w\diamond \lbar{w}'\rc+\lc \lbar{w}\diamond w'\rc+\lambda\lc\lbar{w}\diamond\lbar{w}'
\rc,&\text{ if }w=\lc\lbar{w}\rc,\,w'=\lc\lbar{w}'\rc\in\lc\frak X_{\infty}\rc.
\end{array}\right.
\end{aligned}
\mlabel{eq:Bdia}
\end{equation}
Here the product in the first case is by concatenation and in the second case is by the induction hypothesis on $n$.
Now assume that $\bre(w)\geq1$ or $\bre(w')\geq1$. Let $w=w_{1}\cdots w_{m}$ and $w'=w'_{1}\cdots w'_{m'}$ be the standard decompositions of $w$ and $w'$. Define
\begin{equation}
w\diamond w':=w_{1}\cdots w_{m-1} (w_{m}\diamond w'_{1})w'_{2}\cdots w'_{m'},
\mlabel{eq:cdiam}
\end{equation}
where $w_{m}\diamond w'_{1}$ is defined by Eq.~(\mref{eq:Bdia}) and the other multiplications are
given by the concatenation.

\begin{theorem}\cite{Gub}
The triple $(\bfk \frak X_{\infty}, \diamond, P)$ is the free (noncommutative)
Rota-Baxter algebra generated by $X$.
\mlabel{thm:ncfree}
\end{theorem}

Throughout the rest of the paper, we write
$$\ID:=X \cup \lc \frak X_\infty\rc.
$$
We next show that each Rota-Baxter bracketed word has a unique factorization with respect to the product $\diamond$, of
a particular form specified in the following definition.

\begin{defn}
A sequence $w_1,\cdots,w_m$ from the set $\ID$ is called {\bf alternating} if
no consecutive elements in the sequence are in $\lc \frak X_\infty\rc$.
In other words, for each $1\leq i\leq m-1$, either $w_i$ or $w_{i+1}$ is in
$X$.
\mlabel{de:alt}
\end{defn}

Now we are ready to prove the main result on \cdiam factorizations.
\begin{prop}
For each $w\in \frak X_\infty$, there
exists a unique alternating sequence $w_1,\cdots,w_m$ in $\ID$ such that
\begin{equation}
w =w_{1}\,\diamond\,\cdots\,\diamond\,w_{m},
\mlabel{eq:Exist}
\end{equation}
called the {\bf \diam factorization} of $w$. We also call $\w(w):=m$ the {\bf width} of $w$.
\mlabel{pp:cdiam}
\end{prop}
We note the difference between the breadth and width of a Rota-Baxter word. For example, when $w=x_1x_2$ with $x_1, x_2\in X$, then $b(w)=1$ and $\w(w)=2$. We also note that $w_1,\cdots,w_m$ is alternating means that, for each $1\leq i\leq m-1$,  $w_i$ and $w_{i+1}$ cannot be both in $\lc \frakX_\infty\rc$. However, as showing in the last example, it is okay for $w_i$ and $w_{i+1}$ to be both in $X$.

\begin{proof}
We first prove the existence of the \diam factorization. Let $w = u_1\cdots u_b$ be the standard decomposition of $w$. So  the factors $u_1,\cdots, u_n$ are alternatively in $S(X)$ or $\lc \frakX_\infty\rc$. By expanding the factors that are in $S(X)$, we obtain $w=w_1\cdots w_m$ where $w_1,\cdots,w_m$ is an alternating sequence.
Then $w = w_1\,\diamond\,\cdots \,\diamond\, w_m$ by the definition of $\diamond$ in Eq.~(\mref{eq:Bdia}).

To verify the uniqueness, let $w_{1},\cdots,w_{m}$ and
$w'_{1},\cdots,w'_{m'}$ be two alternating sequences such that
$$w =w_{1}\,\diamond\,\cdots\,\diamond\,w_{m}
=w'_{1}\,\diamond\,\cdots\,\diamond\,w'_{m'}.$$
By Definition \mref{de:alt} and Eq.~(\mref{eq:Bdia}),
$$w_{1}\cdots w_{w} = w_{1} \diamond\cdots\diamond w_{w}
=w =  w'_{1} \diamond\cdots\diamond w'_{w'}
= w'_{1} \cdots w'_{m'}.$$
There is an $n\geq 0$ such that $w\in \frak X_n \subseteq \frakM_n$ and $w_{1}\cdots w_{w} =  w'_{1} \cdots w'_{m'}$ in $\frakM_n$. Since $\frak M_n = M(X\cup \lc \frakM_{n-1}\rc)$ is a free monoid, we have $m=m'$ and $w_{i} = w'_{i}$ for $1\leq i\leq m$.
\end{proof}

\section{The bialgebra structure on free Rota-Baxter algebras}
\mlabel{sec:bial}
In this section, we apply the \cdiam factorization in Proposition~\mref{pp:cdiam} to obtain a coproduct on $\ncsha(X)$ and show that, together with the product $\diamond$, it equips $\ncsha(X)$ with a bialgebra structure. The main Theorem~\mref{thm:main} is stated in Section~\mref{ss:main}. The proof is divided into two steps which are carried out in the remaining two subsections.

\subsection{The coproduct and the statement of the main theorem}
\mlabel{ss:main}
In this subsection, we define a coproduct on $\ncsha(X)$ and prove some of its properties.
We first recall some concepts on bialgebras~\mcite{Gub,LV}.

\begin{defn}
A {\bfk}-bialgebra is a quintuple $(H,\mu,u,\Delta,\vep)$, where $(H,\mu,u)$ is a {\bfk}-algebra and $(H,\Delta,\vep)$ is a {\bfk}-coalgebra such that $\Delta: H\rightarrow H\otimes H$ and $\vep: H\rightarrow{\bfk} $ are algebra homomorphisms. Here the product on $H\ot H$ is given by the tensor product of the one on $H$.
\end{defn}

We will construct $\Delta:\ncsha(X)\to \ncsha(X)\ot \ncsha(X)$ by defining $\Delta(w)$ for $w\in\frak X_{\infty}$ through an induction on the depth $\dep(w)$.
When $\dep(w)=0$, we have $w\in\frak X_{0}=M(X)$. We then define
\begin{equation}
\Delta(w):=1\ot1\,\text{ when }\, w=1
\mlabel{eq:Init}
\end{equation}
and
\begin{equation}
\Delta(w):=x\ot1 + 1\ot x  \,\text{ when }\, w=x\in X.
\mlabel{eq:Dbull}
\end{equation}
Applying this definition, when $ w=x_{1}\cdots x_{m}\in S(X)$ with $m\geq2$ and $x_{i}\in X$ for $1\leq i\leq m$, we set
$$\Delta(w):=\Delta(x_1)\,\diamond\, \cdots \,\diamond\, \Delta(x_{m}).$$

Suppose that $\Delta(w)$ have been defined for $w\in\frak X_{\infty}$ with $\dep(w)\leq n$ and consider $w\in \frak X_{\infty}$ with $\dep(w)=n+1$. Then Proposition~\mref{pp:cdiam} gives $w= w_{1}\cdots w_{m}$ for a unique alternating sequence $w_1,\cdots, w_m, m\geq 1$.
First assume that $m=1$. Then $w$ is in $\lc\frak X_{\infty}\rc$ since $\dep(w)=n+1\geq1$. Write $w:=\lc\lbar{w}\rc$ with $\lbar{w}\in \frak X_{\infty}$. We then define
\begin{equation}
\Delta(w)=\Delta(\lc\lbar{w}\rc):=w\ot1 + (\id\ot P)\Delta(\lbar{w}).
\mlabel{eq:Tree}
\end{equation}
Here $\Delta(\lbar{w})$ is defined by the induction hypothesis.
For general $m\geq 1$, we define
\begin{equation}
\Delta(w):=\Delta(w_{1})\,\diamond\,\cdots\,\diamond\, \Delta(w_{m}),
\mlabel{eq:Forest}
\end{equation}
where $\Delta(w_{1}), \cdots, \Delta(w_{m})$ are defined in Eq.~(\mref{eq:Dbull}) or Eq.~(\mref{eq:Tree}).
By the uniqueness of the \cdiam factorization of $w$, $\Delta(w)$ is well-defined. This completes the inductive construction of $\Delta$.
\smallskip

Let $w$ be an RBW in $\frak X_{\infty}$. Define $\vep: \ncsha(X)\rightarrow\bfk$ by setting
\begin{equation}
\vep(w)=\left \{\begin{array}{ll} 1_\bfk, & \text{if } w=1, \\
0, & \text{otherwise,}
\end{array}
\right.
\mlabel{eq:vep}
\end{equation}
and extending by \bfk-linearity.
The following is our main result of this section.

\begin{theorem}
The quintuple $(\ncsha(X), \diamond, u, \Delta, \vep)$ is a {\bfk}-bialgebra, where
$$u : \bfk\rightarrow\ncsha(X), 1_{\bfk}\mapsto 1.$$
\mlabel{thm:main}
\end{theorem}
\begin{proof}
By Theorem~\mref{thm:ncfree}, the triple $(\ncsha(X),\diamond,u)$ is a unitary $\bfk$-algebra.
Instead of verifying that the triple $(\ncsha(X),\Delta,\vep)$ is a coalgebra directly, we first verify the compatibility conditions of the coproduct $\Delta$ and counit $\vep$ with the product $\diamond$ and the unit $u$, that is, both $\Delta$ and $\vep$ are algebra homomorphisms.

The compatibility conditions with $u$ are easily verified by Eq.~(\mref{eq:Init}). The proofs that $\Delta$ and $\vep$ are compatible with the multiplication $\,\diamond\,$ are given in Proposition~\mref{prop:calgh}
and Lemma~\mref{lem:counitprod} in Section~\mref{ss:comp}, after providing some preliminary results in the rest of this subsection.

With the help of the compatibility conditions, the coalgebra condition is verified in Section~\mref{ss:coalg} (Proposition~\ref{pp:them1}). This completes the proof of the theorem.
\end{proof}

We introduce an auxiliary  notation before stating the next result. For a positive integer $n$, denote $[n]:=\{1,\cdots,n\}$. For a subset $I=\{i_1<\cdots<i_k\}\subseteq [n]$, denote $\arro{x_I}:=x_{i_1}\cdots x_{i_k}$, with the convention that $x_\emptyset=1$.

\begin{lemma}
Let $w=x_{1},\cdots,x_{m}$ with $m\geq1$ and $x_{i}\in X$ for $1\leq i\leq m$. Then
\begin{equation}
\Delta(x)=\sum_{I\sqcup J =[m]}\arro{x_I}\ot \arro{x_J}.
\mlabel{eq:equi}
\end{equation}
\end{lemma}

\begin{proof}
We prove Eq.~(\mref{eq:equi}) by induction on $m\geq1$.
When $m = 1$, Eq.~(\mref{eq:equi}) is valid by Eq.~(\mref{eq:Dbull}).
Assume that  Eq.~(\mref{eq:equi}) holds for $m\geq1$ and consider the case of $m+1$. Let
$w=x_{1}\cdots x_{m+1}$ with $x_{i}\in X$ for $1\leq i\leq m+1$. Then $w=x_{1}\cdots x_{m+1}$ is the standard decomposition of $x$.
Thus
\allowdisplaybreaks{
\begin{align*}
\Delta(w)&=\Delta(x_{1}\cdots x_{m+1})\\
&=\Delta(x_{1}) \,\diamond\, \cdots \,\diamond\, \Delta(x_{m+1})
\quad(\mbox{by Eq.~(\mref{eq:Forest})})\\
&=\Delta(x_{1}\cdots x_{m})\diamond \Delta(x_{m+1})
\quad(\mbox{by Eq.~(\mref{eq:Forest})})\\
&= \left(\sum_{I\sqcup J =[m]}\arro{x_I}\ot \arro{x_J}\right)
\diamond\left(x_{m+1}\ot1+1\ot x_{m+1}\right)\quad(\mbox{by the induction hypothesis})\\
&= \sum_{I\sqcup J =\, [m]} \left(x_{I}\,\diamond\, x_m \right) \ot x_{J}
+
\sum_{I\sqcup J =\, [m]} x_{I} \ot \left(x_{J} \,\diamond\,x_{m}\right)\\
&= \sum_{I\sqcup J =\, [m]} \left(x_{I}x_m \right) \ot x_{J}
+
\sum_{I\sqcup J =\, [m]} x_{I} \ot \left(x_{J}x_{m}\right)\\
&= \sum_{I\sqcup J =\, [m+1]} x_{I} \ot x_{J},
\end{align*}
}
as required.
\end{proof}

\subsection{The compatibility conditions}
\mlabel{ss:comp}
We now verify that both $\Delta$ and $\vep$ are compatible with the multiplication $\,\diamond\,$, with the compatibilities with the unit $u$ being clear. We start with a simple case.

\begin{lemma}
Let $u= x_{1} \cdots x_{p}$ and $v=y_{1} \cdots y_{q}$ with $p,q\geq 0$, $x_i,y_j\in X$, $1\leq i\leq p$ and $1\leq j\leq q$. Then
\begin{equation}
\Delta(u\,\diamond\,v)=\Delta(u)\,\diamond\,\Delta(v).
\notag
\end{equation}
\mlabel{lem:I1}
\end{lemma}

\begin{proof}
If $u = 1$ or $v = 1$, without loss of generality,
let $u = 1$. Then by Eq.~(\mref{eq:Init}) $\Delta(u) = 1\ot 1$
and so
$$\Delta(u\,\diamond\,v)= \Delta(v) = \Delta(u)\,\diamond\,\Delta(v).$$
Suppose that $u \neq 1$ and $v \neq 1$.
Then $p,q\geq 1$.
We denote $z_{i}=x_{i}$ for $1\leq i\leq p$ and $z_{p+j}=y_{j}$ for $1\leq j\leq q$. Then
\allowdisplaybreaks{
\begin{align*}
u\,\diamond\,v&= (x_1\cdots x_p)
\,\diamond\, (y_1\cdots y_q) \\
&= x_1\cdots x_{p-1} (x_p \,\diamond \, y_1) y_2\cdots y_q
\quad(\text{by Eq.~(\mref{eq:cdiam})})\\
&= x_1 \cdots x_p y_1\cdots y_q = z_1 \cdots z_{p+q}.
\end{align*}
}
By Eq.~(\mref{eq:equi}), we have
\begin{equation*}
\Delta(u\,\diamond\,v)=
\sum_{I\sqcup J =\, [p+q]}
z_I\ot z_J.
\end{equation*}
Similarly,
$$\Delta(u)= \sum_{I\sqcup J =\, [p]}
x_I \ot x_J,\quad
\Delta(v)= \sum_{I'\sqcup J' = \, [q]}
y_{I'} \ot y_{J'}.
$$
Thus
\allowdisplaybreaks{
\begin{align*}
\Delta(u)\,\diamond\,\Delta(v)
&= \left(\sum_{I\sqcup J =\, [p]}
x_I \ot x_J \right) \,\diamond\,
\left( \sum_{I'\sqcup J' = \, [q]}
y_{I'} \ot y_{J'}\right)
= \sum_{I\sqcup J =\,[p]\atop I'\sqcup J' =\,[q]}
\Big( x_I \ot x_J \Big) \,\diamond\,
\Big(y_{I'}\ot  y_{J'}\Big) 
\\
&= \sum_{I\sqcup J =\,[p]\atop I'\sqcup J' =\,[q]}
\Big( x_I \,\diamond\, y_{I'} \Big) \ot
\Big(x_J \,\diamond\, y_{J'}\Big)= \sum_{I\sqcup J =\,[p]\atop I'\sqcup J' =\,[q]}
(x_I y_{I'}) \ot
(x_Jy_{J'}) \\
&= \sum_{I\sqcup J =\,[p+q]}
z_I \ot z_J = \Delta(u\,\diamond\,v),
\end{align*}
}
as required.
\end{proof}

Now we can verify the compatibility of $\Delta$ with $\diamond$.

\begin{prop}
Let $u, v\in\ncsha(X)$. Then
\begin{equation}
\Delta(u\,\diamond\, v)=\Delta(u)\,\diamond\,\Delta(v).
\mlabel{eq:Morphism}
\end{equation}
\mlabel{prop:calgh}
\end{prop}

\begin{proof}
Since $\Delta$ is linear and $\diamond$ is bilinear, it is sufficient to verify that Eq.~(\mref{eq:Morphism})
holds for basis elements $u, v\in \frak X_{\infty}$. For this we proceed by induction on the sum of the depths $n:=\dep(u) + \dep(v)$ of $u$ and $v$.
When $n=0$, then $\dep(u)=\dep(v) =0$. Let
$$u=x_{1}\cdots x_{p}\,\text{ and }\,v=y_{1}\cdots y_{q}.$$
Then Eq.~(\mref{eq:Morphism}) holds by Lemma~\mref{lem:I1}.

Assume that Eq.~(\mref{eq:Morphism}) holds for $n\leq k$ with $k\geq 0$ and consider $u, v\in \frak X_{\infty}$ with $n=k+1$. We next reduce to prove Eq.~(\mref{eq:Morphism}) by induction on the sum of the breadths $m:=\bre(u)+\bre(v)$ of $u$ and $v$. Then $m\geq2$.

When $m=2$, then $\bre(u)=\bre(v)=1$. By Eq.~(\mref{eq:Forest}), Eq.~(\mref{eq:Morphism}) holds whenever $u\in X$ or $v\in X$.
So we are left to consider the case when $u:=\lc\lbar{u}\rc$ and $v :=\lc\lbar{v}\rc$, where $\lbar{u},\lbar{v}\in \frak X_{\infty}$ and $n=k+1\geq1$.
Using the Sweedler notation, we write
$$\Delta(\lbar{u})=\sum_{(\lbar{u})}\lbar{u}_{(1)}\otimes \lbar{u}_{(2)} \text{ and }
\Delta(\lbar{v})=\sum_{(\lbar{v})}\lbar{v}_{(1)}\otimes \lbar{v}_{(2)}.$$
Then we obtain
\begin{align*}
&\Delta(u\,\diamond\, v) = \Delta\left(\lc\lbar{u}\rc\,\diamond\, \lc\lbar{v}\rc\right)= \Delta \lc \lbar{u}\,\diamond\, \lc\lbar{v}\rc+\lc\lbar{u}\rc\,\diamond\, \lbar{v}
+\lambda\lbar{u}\,\diamond\,\lbar{v})\rc \\
&= \lc\left(\lbar{u}\,\diamond\, \lc\lbar{v}\rc+\lc\lbar{u}\rc\,\diamond\, \lbar{v}
+\lambda\lbar{u}\,\diamond\,\lbar{v} \right)\rc\ot1
+(\id\ot P)\Delta\left(\lbar{u}\,\diamond\, \lc\lbar{v}\rc+\lc\lbar{u}\rc\,\diamond\, \lbar{v}
+\lambda\lbar{u}\,\diamond\,\lbar{v}\right)\\
&\hspace{7cm} (\text{by Eq.~(\mref{eq:Tree})})\\
&=(u\diamond v)\ot1+(\id\ot  P)
\left(\Delta(\lbar{u}) \,\diamond\,\Delta(\lc\lbar{v}\rc)+\Delta(\lc\lbar{u}\rc)\,\diamond\,\Delta(\lbar{v})
+\lambda\Delta(\lbar{u})\,\diamond\,\Delta(\lbar{v})\right)\\
&\hspace{7cm} (\text{by the induction hypothesis on~}n)\\
&=(u\,\diamond\, v)\ot1+(\id\ot  P)\bigg(\Delta(\lbar{u})
\big(\lc\lbar{v}\rc\ot1+(\id\ot  P)\Delta(\lbar{v})\big)
+\big(\lc\lbar{u}\rc\ot1\\
&\,+(\id\ot P)\Delta(\lbar{u})\big)\,\diamond\,\Delta(\lbar{v})+\lambda\Delta(\lbar{u})
\,\diamond\,\Delta(\lbar{v})\bigg)
\quad(\text{by Eq.~(\mref{eq:Tree})})\\
&=(u\,\diamond\, v)\ot1+(\id\ot  P)\bigg(\sum _{(\lbar{u})}(\lbar{u}_{(1)}
\,\diamond\, \lc\lbar{v}\rc)\ot \lbar{u}_{(2)}+ \sum_{(\lbar{u}),\,(\lbar{v}) }
(\lbar{u}_{(1)}\,\diamond\, \lbar{v}_{(1)})\ot (\lbar{u}_{(2)}\,\diamond\, \lc\lbar{v}_{(2)}\rc)\\
&\,+\sum_{(\lbar{v})}(\lc\lbar{u}\rc\,\diamond\, \lbar{v}_{(1)})\ot \lbar{v}_{(2)}
+\sum_{(\lbar{u}),\,(\lbar{v}) } (\lbar{u}_{(1)}\,\diamond\, \lbar{v}_{(1)})\ot
(\lc\lbar{u}_{(2)}\rc\,\diamond\, \lbar{v}_{(2)})+\lambda\sum_{(\lbar{u}),\,(\lbar{v}) }
(\lbar{u}_{(1)}\,\diamond\, \lbar{v}_{(1)})\ot (\lbar{u}_{(2)}\,\diamond\, \lbar{v}_{(2)})\bigg)\\
&=(u\,\diamond\, v)\ot1+\sum _{(\lbar{u})}(\lbar{u}_{(1)}\,\diamond\, \lc\lbar{v}\rc)
\ot \lc\lbar{u}_{(2)}\rc+\sum_{(\lbar{u}),\,(\lbar{v}) } (\lbar{u}_{(1)}\,\diamond\,
\lbar{v}_{(1)})\ot \lc\lbar{u}_{(2)}\,\diamond\, \lc\lbar{v}_{(2)}\rc\rc\\
&\,+\sum _{(\lbar{v})}(\lc\lbar{u}\rc\,\diamond\, \lbar{v}_{(1)})\ot \lc\lbar{v}_{(2)}\rc
+\sum_{(\lbar{u}),\,(\lbar{v}) } (\lbar{u}_{(1)}\,\diamond\, \lbar{v}_{(1)})\ot
\lc \lc\lbar{u}_{(2)}\rc\,\diamond\, \lbar{v}_{(2)}\rc+\lambda\sum_{(\lbar{u}),\,(\lbar{v}) }
(\lbar{u}_{(1)}\,\diamond\, \lbar{v}_{(1)})\ot \lc\lbar{u}_{(2)}\,\diamond\, \lbar{v}_{(2)}\rc\\
&= (u\,\diamond\, v)\ot1+\sum _{(\lbar{u})}(\lbar{u}_{(1)}\,\diamond\, \lc\lbar{v}\rc)
\ot \lc\lbar{u}_{(2)}\rc+\sum _{(\lbar{v})}(\lc\lbar{u}\rc\,\diamond\, \lbar{v}_{(1)})\ot \lc\lbar{v}_{(2)}\rc\\
&\,+\sum_{(\lbar{u}),\,(\lbar{v}) } (\lbar{u}_{(1)}\,\diamond\, \lbar{v}_{(1)})\ot \lc\lbar{u}_{(2)}\,\diamond\, \lc\lbar{v}_{(2)}\rc+\lc\lbar{u}_{(2)}\rc\,\diamond\, \lbar{v}_{(2)}
+\lambda\lbar{u}_{(2)}\,\diamond\,\lbar{v}_{(2)}\rc
\quad(\text{by the bilinearity of $\diamond$ and $\ot$})\\
&= \left(\lc\lbar{u}\rc\ot1+\sum_{(\lbar{u})}\lbar{u}_{(1)}\ot \lc\lbar{u}_{(2)}\rc\right)\,\diamond\,
\left(\lc\lbar{v}\rc\ot1+\sum_{(\lbar{v})} \lbar{v}_{(1)}\ot \lc\lbar{v}_{(2)}\rc\right)\\
&= \left( \lc\lbar{u}\rc\ot1+(\id\ot  P)
\left( \sum_{(\lbar{u})}\lbar{u}_{(1)}\ot \lbar{u}_{(2)}\right )\right)\,\diamond\,
\left(\lc\lbar{v}\rc\ot1+(\id\ot  P)\left(\sum_{(\lbar{v})}\lbar{v}_{(1)}\ot
\lbar{v}_{(2)}\right)\right )\\
&= \left( \lc\lbar{u}\rc\ot1+(\id\ot  P)\Delta(\lbar{u})\right)\,\diamond\,
\left(\lc\lbar{v}\rc\ot1+(\id\ot  P)\Delta(\lbar{v})\right)\\
&=\Delta(\lc\lbar{u}\rc)\,\diamond\,\Delta(\lc\lbar{v}\rc)
\quad(\text{by Eq.~(\mref{eq:Tree})})\\
&= \Delta(u)\,\diamond\,\Delta(v).
\end{align*}
This completes the initial step of the induction on $\bre(u)+\bre(v)$.

Assume that Eq.~(\mref{eq:Morphism})
holds for  $u,v\in \frak X_{\infty}$ with $n=k+1$ and $2\leq m\leq \ell$ for a $\ell\geq 2$ and consider the case when $u, v\in \frak X_{\infty}$ with $n=k+1$ and $m=\ell+1$. Then $m\geq3$, so either $u$ or $v$ has breadth greater than or equal to $2$. There are three cases to consider: (i) $\bre(u)\geq2,\bre(v)\geq2$; (ii) $\bre(u)\geq2,\bre(v)=1$; (iii) $\bre(u)=1,\bre(v)\geq2$.

\noindent
{\bf Case (i)}. $\bre(u)\geq2,\bre(v)\geq2$. Let $u=u_{1}\cdots u_{p}$ and  $v=v_{1}\cdots v_{q}$, where $u_{1}, \cdots,u_{p}\in \ID$ and $v_{1},\cdots ,v_{q}\in\ID$ are alternating.
Then
$$u\,\diamond\, v = u_{1}\cdots u_{p-1}(u_{p}\diamond v_{1})v_{2}\cdots v_{q}.$$
If $u_{p}\notin \lc \frak X_\infty\rc$ or  $v_1\notin \lc \frak X_\infty\rc$,
then
\begin{align*}
\Delta(u\,\diamond\, v) = & \Delta(u_{1}\,\diamond\, \cdots \,\diamond\,u_{p} \,\diamond\,
v_{1}\,\diamond\, \cdots \,\diamond\,v_{q})\\
=& \Delta(u_{1}) \,\diamond\, \cdots \,\diamond\, \Delta(u_{p}) \,\diamond\,
\Delta(v_{1})\,\diamond\, \cdots \, \Delta(\diamond\,v_{q})\\
=& \Delta(u_{1} \,\diamond\, \cdots \,\diamond\, u_{p}) \,\diamond\,
\Delta(v_{1}\,\diamond\, \cdots \, \diamond\,v_{q})\\
=& \Delta(u_{1} \cdots   u_{p}) \,\diamond\,
\Delta(v_{1}  \cdots  v_{q}) = \Delta(u) \,\diamond\, \Delta(v).
\end{align*}
Suppose $u_{p} \in \lc \frak X_\infty\rc$ and $v_1 \in \lc \frak X_\infty\rc$.
Let
$u_{p}=\lc\lbar{u}_{p}\rc$ and $v_{1}=\lc\lbar{v}_{1}\rc$ for some $\lbar{u}_{p},\lbar{v}_{1}\in \frak X_\infty$. Write
$$u_{p}\diamond v_{1}= \sum_{i} c_i \lc w_{i}\rc ,\text{ where }\, c_i\in \bfk, u_{i}\in\frak X_{\infty}.$$
Note that $u_{1}\cdots u_{p-1}\lc w_{i}\rc v_{2},\cdots,v_{q}$ is alternating.
So we have
\begin{align*}
\Delta(u\,\diamond\, v)
&=\Delta(u_{1}\cdots u_{p-1}(u_{p}\,\diamond \, v_{1})v_{2}\cdots v_{q})=\sum_{i}c_i \Delta(u_{1}\cdots u_{p-1}\lc w_{i}\rc v_{2} \cdots v_{q})\\
=& \sum_{i}c_i \Delta(u_{1} \,\diamond \, \cdots \,\diamond \, u_{p-1} \,\diamond \, \lc w_{i}\rc
 \,\diamond \, v_{2} \,\diamond \, \cdots \,\diamond \, v_{q})\\
&=\sum_{i}c_i \Delta(u_{1})\diamond\cdots\diamond \Delta(u_{p-1})\diamond\Delta(\lc w_{i}\rc)\diamond\Delta( v_{2})\diamond \cdots\diamond \Delta(v_{q})\\
&=\Delta(u_{1})\diamond\cdots\diamond \Delta(u_{p-1})\diamond\Delta(u_{p}\diamond v_{1})\diamond\Delta( v_{2})\diamond \cdots\diamond \Delta(v_{q})\\
&=\Delta(u_{1})\diamond\cdots\diamond \Delta(u_{p-1})\diamond\Delta(u_{p})\diamond \Delta(v_{1})\diamond\Delta( v_{2})\diamond \cdots\diamond \Delta(v_{q})\quad(\text{by the case when } m=2)\\
&=\Delta(u_{1}\diamond\cdots\diamond u_{p})\diamond \Delta(v_{1}\diamond\cdots\diamond v_{q})\\
&=\Delta(u)\diamond\Delta(v).
\end{align*}

The Cases (ii) and (iii) are easier than Case~(i) and can be similarly checked.

This completes the induction on $\bre(u)+\bre(v)=m$ and hence the induction on $\dep(u)+\dep(v)=n$, which in turns completes the proof of Proposition~\mref{prop:calgh}.
\end{proof}

The following result shows that $\varepsilon$ defined by Eq.~(\mref{eq:vep}) is an algebra morphism.

\begin{lemma}
Let $u, v\in\ncsha(X)$. Then
\begin{equation}
\varepsilon(u\,\diamond\, v)=\varepsilon(u)\varepsilon(v).
\mlabel{eq:Morphism'}
\end{equation}
\mlabel{lem:counitprod}
\end{lemma}

\begin{proof}
By the linearity of $\varepsilon$, we just need to verify Eq.~(\mref{eq:Morphism'}) for $u, v\in \frak X_\infty$.
If $u=1$ or $v=1$, without loss of generality, letting $u=1$, then
$$\varepsilon(u\,\diamond\, v)
=\varepsilon(v)
=\varepsilon(u)\varepsilon(v).$$
If $u\neq1$ and $v\neq1$, then $u\,\diamond\,v\neq1$. So by Eq.~(\mref{eq:vep}), both sides of Eq.~(\mref{eq:Morphism'}) are zero, as needed.
\end{proof}

\subsection{The coalgebra structure}
\mlabel{ss:coalg}
We now establish the coalgebra structure on $\ncsha(X)$ and hence finish the proof of Theorem~\mref{thm:main}.

\begin{prop} Let $X$ be a nonempty set.
The triple $(\ncsha(X), \Delta, \varepsilon)$ is a coalgebra.
\mlabel{pp:them1}
\end{prop}

\begin{proof}
We only need to verify the coassociativity and the counicity for $w\in \frak X_\infty$.
We first apply the induction on the depth $n:=\dep(w)$ of the $w\in \frakX_\infty$ to verify the coassociativity of
$\Delta$, namely, the equality
\begin{equation}
(\Delta\ot\id)\Delta(w)=(\id\ot\Delta)\Delta(w)\, \text{ for }\, w\in \frak X_\infty.
\mlabel{eq:coass}
\end{equation}
When $n=0$, we have $w= x_{1}\cdots x_{m}$ with $x_i\in X, 1\leq i\leq m, m\geq 0$.
By Eq.~(\mref{eq:equi}), we obtain
\begin{equation*}
\Delta(w)= \sum_{I\sqcup J = \, [m]}
x_I\ot x_J .
\end{equation*}
So we obtain
\begin{align*}
(\id\otimes\Delta)\Delta(w)&= (\id\otimes\Delta) \left( \sum_{I\sqcup J = \, [m-1]}
x_I\ot x_J\right)= \sum_{I\sqcup J = \, [m]}
x_I \ot \Delta(x_J) \\
&= \sum_{I\sqcup J = \, [m]}
x_I \ot \left(\sum_{J_1\sqcup J_2=J}x_{J_1}\ot x_{J_2}\right) = \sum_{I\sqcup J_1 \sqcup J_2 = \, [m]}
x_I \ot x_{J_1} \ot x_{J_2}.
\end{align*}
Similarly, we obtain
\begin{align*}
(\Delta\otimes\id)\Delta(w) &= (\Delta\otimes\id) \left( \sum_{I\sqcup J = \, [m]}
x_I\ot x_J\right)=  \sum_{I\sqcup J = \, [m]}
\Delta(x_I)\ot x_J \\
&= \sum_{I_1\sqcup I_2 \sqcup J = \, [m]}
x_{I_1} \ot x_{I_2} \ot x_J = (\id\otimes\Delta)\Delta(w).
\end{align*}
This completes the initial step.

Assume that $\Delta$ is coassociative for $w\in \frak X_\infty$ with $\dep(w)=n$ and consider
$w\in \frak X_\infty$ with $\dep(w)=n+1$. Let $w=w_{1}\,\diamond\,\cdots\,\diamond\,w_{m}$ be the alternating factorization of $w$ with width $m\geq1$. We next apply induction on $m$.
When $m=1$, since $n+1\geq 1$, we can write $w=P(\overline{w})\in \frak X_{n+1}$ with $\dep(\overline{w})=n$. So we have
\allowdisplaybreaks{
\begin{align*}
(\id\ot \Delta)\Delta(w)
&=(\id\ot \Delta)\Delta(P(\overline{w}))\\
&=(\id\ot\Delta)(w\ot1+(\id\ot P)\Delta(\overline{w}))
\quad (\text{by Eq.~(\mref{eq:Tree})})\\
&=w\ot1\ot 1+(\id\ot (\Delta P))\Delta(\overline{w})\quad (\text{by Eq.~(\mref{eq:Init})})\\
&=w\ot1\ot1+(\id\ot P)\Delta(\overline{w})\ot1
+(\id\ot\id\ot P)(\id\ot\Delta)\Delta(\overline{w}) \\
&\hspace{6cm} (\text{by Eq.~(\mref{eq:Tree})})\\
&=w\ot1\ot1+(\id\ot P)\Delta(\overline{w})\ot1
+(\id\ot\id\ot P)(\Delta\ot\id)\Delta(\overline{w})\\
&\hspace{6cm} (\text{by the induction hypothesis on~}n)\\
&=w\ot1\ot1+ (\id\ot P)\Delta(\overline{w})\ot1
+(\Delta\ot P)\Delta(\overline{w})\\
&= (w\ot1+(\id\ot P)\Delta(\overline{w}) )\ot1
+(\Delta\ot P)\Delta(\overline{w}) \\
&=\Delta(w)\ot1+(\Delta\ot P)\Delta(\overline{w})
\quad(\text{by Eq.~(\mref{eq:Tree})})\\
&=(\Delta\ot\id)(w\ot1+(\id\ot P)\Delta(\overline{w}))\\
&=(\Delta\ot\id)\Delta(w).\quad (\text{by Eq.~(\mref{eq:Tree})})
\end{align*}
}
To summarize, Eq.~(\mref{eq:coass}) holds for $w\in \frak X_\infty$ with $\dep(w)=n+1$ and $\w(w)=1$. Assume that
Eq.~(\mref{eq:coass}) holds for $w\in \frak X_\infty$ with $\dep(w)=n+1,\w(w)\leq m$ for $m\geq1$ and consider $w\in \frak X_\infty$ with $\dep(w)=n+1$ and $\w(w)=m+1$.
Let $w=w_1\,\diamond, \cdots\,\diamond\,w_{m+1}$ be the alternating factorization of $w$. Then we have the alternating factorization $w':=w_2\,\diamond\,\cdots\,\diamond\,w_{m+1}$ and $w=w_1\,\diamond\,\diamond w'$ with $\w(w_{1})=1, \w(w')=m$. Let
$$\Delta(w_{1}):=\sum^{}_{(w_{1})}w_{1(1)}\ot w_{1(2)}\,\text{ and }\,
\Delta(w'):=\sum^{}_{(w')}w'_{(1)}\ot w'_{(2)}.$$
By the induction hypothesis on the width, we obtain
$$(\Delta\ot\id)\Delta(w_{1})
=(\id\ot\Delta)\Delta(w_{1})\,\text{ and }\,(\Delta\ot\id)\Delta(w')
=(\id\ot\Delta)\Delta(w').$$
In other words,
\begin{equation}
\sum^{}_{(w_{1})}\Delta(w_{1(1)})\ot w_{1(2)}
=\sum^{}_{(w_{1})}w_{1(1)}\ot\Delta(w_{1(2)})\,\text{ and }\,
\sum^{}_{(w')}\Delta(w'_{(1)})\ot w'_{(2)}
=\sum^{}_{(w')}w'_{(1)}\ot\Delta( w'_{(2)}).
\mlabel{eq:coass0}
\end{equation}
Then
\begin{align*}
(\id\ot\Delta)\Delta(w)
&=(\id\ot\Delta)\Delta(w_{1}\,\diamond\, w')\\
&=(\id\ot\Delta)(\Delta(w_{1})\,\diamond\,\Delta(w'))
\quad(\text{by Eq.~(\mref{eq:Morphism})})\\
&=\sum^{}_{(w_{1})}\sum^{}_{(w')}(w_{1(1)}\,\diamond\, w'_{(1)})\ot\Delta(w_{1(2)}\,\diamond\, w'_{(2)})
\quad(\text{by the linearity})\\
&=\sum^{}_{(w_{1})}\sum^{}_{(w')}(w_{1(1)}\,\diamond\, w'_{(1)})\ot(\Delta(w_{1(2)})\,\diamond\,\Delta(w'_{(2)}))
\quad(\text{by Eq.~(\mref{eq:Morphism})})\\
&=\left(\sum^{}_{(w_{1})}w_{1(1)}\ot\Delta(w_{1(2)})\right)\,\diamond\,
\left(\sum^{}_{(w')}w'_{(1)}\ot\Delta( w'_{(2)})\right).
\end{align*}

By the same argument, we obtain
$$ (\Delta\ot\id)\Delta(w) =
\left(\sum^{}_{(w_{1})}\Delta(w_{1(1)})\ot w_{1(2)}\right)\,\diamond\,
\left(\sum^{}_{(w')}\Delta( w'_{(1)})\ot w'_{(2)}\right).$$
Thus Eq.~(\mref{eq:coass}) holds by Eq.~(\mref{eq:coass0}).
This completes the inductive proof of Eq.~(\mref{eq:coass}).

We next apply the induction on the depth $n:=\dep(w)$ of $w\in \frakX_\infty$ to check the counicity conditions:
\begin{equation}
(\vep\ot\id)\Delta(w)=\beta_{\ell}(w)\,\text{ and }\, (\id\ot\vep)\Delta(w)=\beta_{r}(w),
\mlabel{eq:coun}
\end{equation}
where $\beta_\ell:\ncsha(X)\rightarrow \bfk\ot\ncsha(X)$ is given by $w\mapsto 1_{\bfk}\ot w$ and
$\beta_{r}:\ncsha(X)\rightarrow \ncsha(X)\ot\bfk $ is given by $w\mapsto w\ot 1_{\bfk}$.

When $n=0$, we have $w= x_{1} \cdots x_{m}$ with $m\geq 0$. Then
\begin{equation*}
\Delta(w)= \sum_{I\sqcup J = [m] }
x_I\ot x_J.
\end{equation*}
Thus we obtain
\begin{align*}
(\vep\ot\id)\Delta(w)&=(\vep\ot\id)
\left(\sum_{I\sqcup J = [m] }
x_I\ot x_J \right) = \sum_{I\sqcup J = [m] }
\vep(x_I)\ot x_J\\
&=\vep(1)\ot  x_{1} \cdots x_{m} \quad (\text{by Eq.~(\mref{eq:vep})}) \\
&=1_{\bfk}\ot w =\beta_\ell(F).
\end{align*}

Assume that Eq.~(\mref{eq:coun}) holds for $w\in \frak X_\infty$ with $\dep(w)=n$ and consider the case of $w\in \frak X_\infty$ with $\dep(w)=n+1$. We next reduce the width of $w$. When $w=1$, we can write $w=P(\overline{w})\in \frak X_{n+1}$ with $\dep(\overline{w})=n$.
Then
\begin{align*}
(\vep\ot \id )\Delta(w)&= (\vep\ot\id )\Delta(P(\overline{w}))\\
&= (\vep\ot \id )(w\ot 1 + (\id\ot P)\Delta(\overline{w}))
\quad(\text{by Eq.~(\mref{eq:Tree})})\\
&= (\vep\ot \id )(\id\ot P)\Delta(\overline{w})
\quad (\text{by Eq.~(\mref{eq:vep})})\\
&=(\id\ot P)(\vep\ot \id )\Delta(\overline{w})\\
&= (\id\ot P)(1_{\bfk}\ot\overline{w})
\quad( \text{by the induction hypothesis on~}n)\\
&= 1_{\bfk}\ot w =\beta_\ell(w).
\end{align*}
Assume that
Eq.~(\mref{eq:coun}) holds for $w\in \frak X_\infty$ with $\dep(w)=n+1,\w(w)\leq m$ for $m\geq1$ and consider
$w\in \frak X_\infty$ with $\dep(w)=n+1,\w(w)=m+1$. As in the proof of the coassiciativity, let $w=w_{1}\,\diamond\, w'$ be from the alternating factorization of $w$ with $\w(w_{1})=1,\w(w')=w$ and
$$\Delta(w_{1}):=\sum^{}_{(w_{1})}w_{1(1)}\ot w_{1(2)}\,\text{ and }\,
\Delta(w'):=\sum^{}_{(w')}w'_{(1)}\ot w'_{(2)}.$$
By the induction on $m$, we obtain
\begin{equation*}
(\vep\ot\id)\Delta(w_{1})=\beta_{\ell}(w_{1})\,\text{ and }\, (\vep\ot\id)\Delta(w')=\beta_{\ell}(w').
\end{equation*}
In other words,
\begin{eqnarray}
\sum^{}_{(w_{1})}\vep(w_{1(1)})\ot w_{1(2)}=1_{\bfk}\ot w_{1}\text{ and }
\sum^{}_{(w')}\vep(w'_{(1)})\ot w'_{(2)}=1_{\bfk}\ot w'.
\mlabel{eq:counit0}
\end{eqnarray}
Thus we have
\begin{align*}
(\vep\ot\id)\Delta(w)&=(\vep\ot\id)\Delta(w_{1}\,\diamond\, w')\\
&=(\vep\ot\id)(\Delta(w_{1})\,\diamond\,\Delta(w'))
\quad(\text{by Eq.~(\mref{eq:Morphism})})\\
&=\sum_{(w_{1}),\, (w')} \vep(w_{1(1)}\,\diamond\, w'_{(1)})\ot(w_{1(2)}\,\diamond\, w'_{(2)})\\
&=\sum_{(w_{1}),\,(w')}
(\vep(w_{1(1)})\vep(w'_{(1)}))\ot(w_{1(2)}\,\diamond\, w'_{(2)})\quad(\text{by Eq.~(\mref{eq:Morphism'})})\\
&=\left(\sum^{}_{(w_{1})}\vep(w_{1(1)})\ot w_{1(2)}\right)
\left(\sum^{}_{(w')}\vep(w'_{(1)})\ot w'_{(2)}\right)\\
&=(1_{\bfk}\ot w_{1})(1_{\bfk}\ot w')\quad(\text{by Eq.~(\mref{eq:counit0})})\\
&=1_{\bfk}\ot(w_{1}\,\diamond\,w')=1_{\bfk}\ot w =\beta_{\ell}(w).
\end{align*}
Similarly, we obtain that $(\id\ot\vep)\Delta(w)=\beta_{r}(w)$ holds for $w\in \frak X_\infty$, completing the proof.
\end{proof}

\section{The Hopf algebra structure}
\mlabel{sec:hopf}

In this section, we discuss the Hopf algebra structure on the free Rota-Baxter algebra $\ncsha(X)$, hereby given by \A forests. Since the conclusions depend on the weight of the Rota-Baxter algebra, we reactive the subscript $\lambda$ in $\sha_\lambda^{\rm NC}(X)$ for distinction. We show that the free Rota-Baxter algebra $\ncshao(X)$ is a connected bialgebra with grading given by a suitable degree, and hence is a Hopf algebra. However, when the weight $\lambda$ is non-zero, this grading does not give a connected bialgebra. So we cannot yet conclude that $\sha_\lambda^{\rm NC}(X)$ is a Hopf algebra. See the remarks at the end of the section.

Recall that a \bfk-bialgebra $(H,\mu,u,\Delta,\vep)$ is called a {\bf graded bialgebra} if there are {\bfk}-modules $H^{(n)}, n\geq0$, of $H$ such that
\begin{enumerate}
\item
$H=\bigoplus\limits^{\infty}_{n=0}H^{(n)}$;
\mlabel{it:It1}
\item
$H^{(p)}H^{(q)}\subseteq H^{(p+q)}$;
\mlabel{it:It2}
\item
$\Delta(H^{(n)})\subseteq\bigoplus\limits^{}_{p+q=n}H^{(p)}\otimes H^{(q)}$;
\mlabel{it:It3}
\end{enumerate}
where $p,q\geq0$. $H$ is called {\bf connected} if in addition $H^{(0)}={\bfk}$.

Let $(C,\Delta,\vep )$ and $(A, \mu, u)$ be any {\bfk}-coalgebra and {\bfk}-algebra respectively and $\Hom(C,A)$ the set of {\bfk}-linear maps from $C$ to $A$. The triple $(\Hom(C,A), \ast, u\vep)$ is called the {\bf convolution algebra}, where the {\bf convolution product} $\ast$ is defined by $f\ast g:=\mu(f\otimes g)\Delta$ for $f,g$ in $\Hom(C,A)$ for which $u\vep$ is the unit. Let $(H, \mu, u, \Delta, \vep)$ be a {\bfk}-bialgebra. A {\bfk}-linear endomorphism of $H$ is called an {\bf antipode} for $H$ if it is the inverse of $\id_{H}$ under the product $\ast$, that is, $S\ast\id_{H}=\id_{H}\ast S=u\vep$.
A {\bf Hopf algebra} is a bialgebra $H$ with an antipode $S$.

As is well-known (see~\cite{Gub}), for example),
a connected bialgebra $(H, \mu, u, \Delta, \vep)$ is a Hopf algebra.

Let $w\in \frak X_\infty$. Define
\begin{equation}
\deg(w):= \deg_P(w) + \deg_X(w),
\mlabel{eq:dDeg}
\end{equation}
where $\deg_P(w)$ (resp. $\deg_X(w)$) denotes the number of occurrences of $P$ (resp. $x\in X$) in $w$. We call $\deg(w)$ the {\bf total degree} of $w$. Then from its definition, we obtain
\begin{equation}
\deg(w)=\sum^{m}_{k=1}\deg(w_{k}),
\mlabel{eq:Deg1}
\end{equation}
for $w=w_{1} \cdots w_{m} \in \frak X_\infty$.

We now give a grading of $H_{{\rm RB}}:=\ncshao(X)$ by define
\begin{equation}
H_{{\rm RB}}^{(n)}:=\bfk\{w\in \frak X_\infty\mid \deg(\frakf)=n\}, \text{ where } n\geq0.
\mlabel{eq:grading}
\end{equation}
Then
\begin{equation}
H_{{\rm RB}}=\bigoplus\limits^{\infty}_{n=0}H_{{\rm RB}}^{(n)}\,,\, H_{{\rm RB}}^{(0)}= \bfk\,\text{ and }\,P(H_{{\rm RB}}^{(n)})\subseteq H_{{\rm RB}}^{(n+1)}.
\mlabel{eq:subset}
\end{equation}
Here the inclusion follows from $\deg(P(w)) = \deg(w)+1$ for $w\in \frak X_\infty$.

\begin{theorem}
The free Rota-Baxter algebra $H_{{\rm RB}}=\ncshao(X)$ of weight zero is a connected bialgebra and hence a Hopf algebra.
\mlabel{thm:Hfree}
\end{theorem}
\begin{proof}
By Eq.~(\mref{eq:subset}), we only need to verify Items~(b) and (c) in the definition of a graded bialgebra. This will be accomplished in Proposition~\mref{pp:mgrad} and \mref{pp:cmgrad} respectively.
\end{proof}

\begin{prop}
Let $H_{{\rm RB}} = \ncshao(X)$ and $p,q\geq0$. Then
\begin{equation}
H_{{\rm RB}}^{(p)}\, \diamond\, H_{{\rm RB}}^{(q)}\subseteq  H_{{\rm RB}}^{(p+q)}.
\mlabel{eq:mgrad}
\end{equation}
\mlabel{pp:mgrad}
\end{prop}

\begin{proof}
 Let $w:=w_{1} \cdots w_{m}\in H_{{\rm RB}}^{(p)}$ and
$w':=w'_{1} \cdots w'_{m'}\in H_{{\rm RB}}^{(q)}$
be two base elements in $\frak X_\infty$ with $\bre(w)=m$, $\bre(w')=m'$ and $m,m'\geq 0$.
If $m=0$ or $m'=0$, without loss of generality, letting $m=0$, then $w=1$ and so
$w\,\diamond\, w' = w' \in H_{{\rm RB}}^{(q)} = H_{{\rm RB}}^{(p+q)}$.
Suppose $m, m'\geq 1$. We now verify Eq.~(\mref{eq:mgrad}) by induction on the sum of degrees $p+q\geq 0$. When $p+q=0$, then $p=q=0$. By Eq.~(\mref{eq:subset}), we obtain that $w=w'=1$ and so $w\,\diamond\,w'=1\in H_{{\rm RB}}^{(0)}$. This finishes the initial step.

Assume that Eq.~(\mref{eq:mgrad}) holds for $w,w'$ with $p+q=s$ where $s\geq 0$ and consider the case when $p+q=s+1$. Under this assumption, we prove Eq.~(\mref{eq:mgrad}) by induction on the sum $=m+m'\geq 2$.
When $m+m'=2$, then $m=m'=1$. If $w\in X$ or $w'\in X$, without loss of generality, letting
$w\in X$, then $p=1$ and
$$w\,\diamond\, w' = w w' \in H_{{\rm RB}}^{(q+1)} = H_{{\rm RB}}^{(p+q)}.$$
So we are left to consider the case when
$$w:=w_{1}=P(\lbar{w}_{1})\,\text{ and }\, w':=w'_{1}=P(\lbar{w}'_{1}),$$
where $$\lbar{w}_{1},\lbar{w}'_{1}\in \frak X_\infty\,\text{ and }\,\deg(w_{1})+\deg(w'_{1})=s+1.$$
Then
\begin{align*}
\deg(w_{1})+\deg(\lbar{w}'_{1})&=\deg(w_{1})+\deg(w'_{1})-1=s=
\deg(\lbar{w}_{1})+\deg(w'_{1}).
\end{align*}
So by the induction hypothesis on $p+q$, we have
\begin{align*}
w_{1}\,\diamond\,\lbar{w}'_{1}
&\in H_{{\rm RB}}^{(\deg(w_{1})+\deg(\lbar{w}'_{1}))}= H_{{\rm RB}}^{(s)},\quad
\lbar{w}_{1}\,\diamond\,w'_{1}
\in H_{{\rm RB}}^{(\deg(\lbar{w}_{1})+\deg(w'_{1}))}= H_{{\rm RB}}^{(s)}.
\end{align*}
By Eq.~(\mref{eq:subset}), we further have
\begin{align*}
P(w_{1}\,\diamond\,\lbar{w}'_{1})&\in P(H_{{\rm RB}}^{(s)})\subseteq H_{{\rm RB}}^{(s+1)},\quad
P(\lbar{w}_{1}\,\diamond\,w'_{1})\in P(H_{{\rm RB}}^{(s)})\subseteq H_{{\rm RB}}^{(s+1)},
\end{align*}
and hence (by Eq.~(\mref{eq:Bdia}))
$$w\,\diamond\,w'_{1}=P(w_{1}\,\diamond\,\lbar{w}'_{1})
+P(\lbar{w}_{1}\,\diamond\,w'_{1})\in H_{{\rm RB}}^{(s+1)}.$$

Assume that Eq.~(\mref{eq:mgrad}) holds when $p+q=s+1$ and $m+m'=t\geq2$ and consider the case when $p+q=s+1$ and $m+m'=t+1\geq3$. There are three cases to consider: (i) $\bre(w)\geq 2$ and $\bre(w')\geq 2$; (ii) $\bre(w)\geq 2, \bre(w')=1$; (iii) $\bre(w)=1, \bre(w')\geq 2$.
\smallskip

\noindent
{\bf Case (i)}. $\bre(w)\geq2$ and $\bre(w')\geq2$. Let $w:=w_{1,1}w_{1,2}$ and $w':=w'_{1,1} w'_{1,2}$,
where $w_{1,1},w_{1,2},
w'_{1,1},w'_{1,2}\in \frak X_\infty$ with breadths $\bre(w_{1,2})=\bre(w'_{1,1})=1$.
Then Eq.~(\mref{eq:cdiam}) gives
$$w\,\diamond\,w' =(w_{1,1} w_{1,2})\,\diamond\,(w'_{1,1} w'_{1,2})
=w_{1,1} (w_{1,2}\,\diamond\,w'_{1,1}) w'_{1,2}.$$
By the induction hypothesis on $t$, we have
$$w_{1,2}\,\diamond\,w'_{1,1}\in H_{{\rm RB}}^{(\deg(w_{1,2})+\deg(w'_{1,1}))}.
$$
So we obtain
\begin{align*}
w\,\diamond\,w'
=w_{1,1} (w_{1,2}\,\diamond\,w'_{1,1}) w'_{1,2}
\in w_{1,1}\,H_{{\rm RB}}^{(\deg(w_{1,2})+\deg(w'_{1,1}))}\,w'_{1,2}
\subseteq H_{\rm RB}^{\deg(w_{1,1})+\deg(w_{1,2})+\deg(w'_{1,1})
+\deg(w'_{1,2})}
\end{align*}
which is $H_{\rm RB}^{\deg(w)+\deg(w')}= H_{{\rm RB}}^{(s+1)}$.

Cases (i) and (ii) can be treated in the same way. This finishes the induction on $m+m'$ and hence the proof of the proposition.
\end{proof}

\begin{prop}
Let $H_{{\rm RB}} = \ncshao(X)$ and $n\geq0$. Then
$\Delta(H_{{\rm RB}}^{(n)})\subseteq\bigoplus\limits^{}_{p+q=n}H_{{\rm RB}}^{(p)}\ot H_{{\rm RB}}^{(q)}.
$
\mlabel{pp:cmgrad}
\end{prop}

\begin{proof}
We verify
\begin{equation}
\Delta(w)\subseteq\bigoplus\limits^{}_{p+q=n}H_{{\rm RB}}^{(p)}\ot H_{{\rm RB}}^{(q)},
\mlabel{eq:cmgrad1}
\end{equation}
for $w\in H_{\rm RB}$ by induction on $\deg(w)\geq 0$. When $\deg(w)=0$, then $w=1$, so Eq.~(\mref{eq:cmgrad1}) holds by Eq.~(\mref{eq:Init}). When $\deg(w)=1$, then Eq.~(\mref{eq:cmgrad1}) holds by Eq.~(\mref{eq:equi}). Assume that Eq.~(\mref{eq:cmgrad1}) holds for $ w$ with $\deg(w)=n\geq 1$ and consider $ w$ with $\deg(w)=n+1$, in which case we verify Eq.~(\mref{eq:cmgrad1}) by induction on the breadth $m:=\bre( w)$ of $ w\in H_{RB}^{(n+1)}$.

If $m=1$, then we can write $ w=P(\lbar{ w})$ for some $\lbar{ w}\in \frak X_\infty$ since $\deg( w)=n+1\geq1$. By the induction hypothesis on $n$, we have
$$\Delta(\lbar{ w})\in \bigoplus\limits^{}_{p+q=n}H_{{\rm RB}}^{(p)}\ot H_{{\rm RB}}^{(q)}.$$
By Eq.~(\mref{eq:Tree}), we have
$$\Delta( w)=\Delta(P(\lbar{ w}))= w\ot 1+(\id\ot P)\Delta(\lbar{ w})\in \bigoplus\limits^{}_{p+q=n+1}H_{{\rm RB}}^{(p)}\ot H_{{\rm RB}}^{(q)}.$$

Assume that Eq.~(\mref{eq:cmgrad1}) holds for $ w$ with $\deg(w)=n+1$ and $\bre( w)=m\geq 1$ and consider the case when $\deg(w)=n+1$ and $\bre(w)=m+1$.
From the alternating factorization of $w$, we obtain $ w=uv=u\diamond v$, where $u, v\in \frak X_\infty$ with $1\leq \bre( u),\bre(v)\leq m$.
%
So by the induction hypothesis on $\deg(w)$, we obtain
$$\Delta(u)\in \bigoplus\limits^{}_{p+q=\deg(u)}H_{{\rm RB}}^{(p)}\ot H_{{\rm RB}}^{(q)}, \Delta(v)\in \bigoplus\limits^{}_{p'+q'=\deg(v)}H_{{\rm RB}}^{(p')}\ot H_{{\rm RB}}^{(q')}.$$
Then we have
\begin{align*}
\Delta( w)&=\Delta( w_{1})\,\diamond\, \Delta( w_{2})\\
&\in\left(\bigoplus\limits^{}_{p+q=\deg(u)}H_{{\rm RB}}^{(p)}\ot H_{{\rm RB}}^{(q)}\right)\,\diamond\,
\left(\bigoplus\limits^{}_{p'+q'=\deg(v)}H_{{\rm RB}}^{(p')}\ot H_{{\rm RB}}^{(q')}\right)\\
&\subseteq \bigoplus\limits^{}_{p+q+p'+q'=\deg(u)+\deg(v)}\left(H_{{\rm RB}}^{(p)}\,\diamond\,H_{\rm RB}^{(p')}\right)\ot \left(H_{{\rm RB}}^{(q)}\,\diamond\,H_{\rm RB}^{(q')}\right) \\
&\subseteq
\bigoplus\limits^{}_{p''+q''=n+1}H_{{\rm RB}}^{(p'')}\ot H_{{\rm RB}}^{(q'')},
\end{align*}
since
$\deg( u)+\deg(v)=\deg( w)=n+1$ by Eq.~(\mref{eq:Deg1}).
This completes the induction.
\end{proof}

We end the paper with remarks on free Rota-Baxter algebra $\sha_\lambda^{\rm NC}(X)$ of weight $\lambda\neq0$. In this case, the following example shows that conditions~(b) and (c) in the definition of a graded bialgebra both fail. Thus the method for the weight zero case cannot be applied to prove that $\sha_\lambda^{\rm NC}(X)$ is a Hopf algebra. Note that in the commutative case, the free Rota-Baxter algebra of any weight is a Hopf algebra~\cite{EG1}, thanks to the Hopf algebra structure on the quasi-shuffle algebra of Hoffman~\mcite{Ho}. Thus it is reasonable to expect that one can prove that the free (noncommutative) Rota-Baxter algebra is a Hopf algebra by some other methods.

\begin{exam}
Consider $\sha_\lambda^{\rm NC}(X)$ with $\lambda\neq 0$. Let $w:= P(1)$. Then
\begin{eqnarray*}
w\,\diamond\,w = P(1)\,\diamond \, P(1)
=2 P^2(1)  + \lambda P(1).
\end{eqnarray*}
%
%
Thus condition (b) for a graded bialgebra fails for $\sha_\lambda^{\rm NC}(X)$.

Further, we have
$$\Delta(w) = \Delta P(1) = P(1) \ot 1 + (\id\ot P)\Delta(1) = P(1) \ot 1 + 1\ot P(1)$$
and thus
\begin{align*}
\Delta(wxw)
=& \Delta(w\,\diamond\,x\,\diamond\,w)
=\Delta(w)\,\diamond\,\Delta(x)\,\diamond\,\Delta(w)\\
=& (P(1) \ot 1 + 1\ot P(1)) \diamond (x\ot 1 + 1\ot x) \diamond  (P(1) \ot 1 + 1\ot P(1))\\
=& (P(1)\,\diamond\,x\,\diamond\,P(1)) \ot 1 + (P(1)\,\diamond\,P(1)) \ot x +
 (x\,\diamond\,P(1)) \ot P(1) + P(1) \ot (P(1) \,\diamond\, x)\\
 &+ (P(1) \,\diamond\, x) \ot P(1) + P(1) \ot (x\,\diamond\,P(1)) + x\ot (P(1) \,\diamond\,P(1))
 + 1 \ot (P(1)\,\diamond\,x\,\diamond\,P(1))\\
 =&  (P(1) x P(1)) \ot 1 + 2P^2(1)\ot x + \lambda P(1) \ot x +
 (x P(1)) \ot P(1) + P(1) \ot (P(1)   x)\\
 &+ (P(1)   x) \ot P(1) + P(1) \ot x P(1) + 2x\ot P^2(1) + \lambda x\ot P(1)
 + 1 \ot (P(1) x P(1)).
\end{align*}
For the third and the second to the last terms, we have
$$\deg(P(1)) + \deg(x) = 1+1 =2\neq3=\deg(wxw),$$ which shows that the bialgebra $\sha_\lambda^{\rm NC}(X)$ is not cograded, that is, condition (c) for a graded bialgebra also fails.
\end{exam}

\smallskip

\noindent {\bf Acknowledgements}: This work was supported by the National Natural Science Foundation of
China (Grant No.~11371177 and 11371178) and the National Science Foundation of US (Grant No. DMS~1001855).

\end{document}